\documentclass[12pt]{amsart}
\usepackage{amsfonts,latexsym,amsthm,amssymb,xspace}
\usepackage{amscd}
\usepackage{amsmath}
\usepackage{color}
\usepackage[cp1251]{inputenc}
\usepackage{amsmath,amsthm,amsfonts,color}
\theoremstyle{plain}
\newtheorem{theor}{Theorem}
\theoremstyle{plain}
\newtheorem{prop}{Proposition}
\theoremstyle{plain}
\newtheorem{lemma}{Lemma}
\theoremstyle{plain}
\newtheorem{cor}{Corollary}
\theoremstyle{remark}

\theoremstyle{remark}

\begin{document}

\title[The structure of subspaces in Orlicz spaces between $L^1$ and $L^2$]{The structure of subspaces in Orlicz spaces between $L^1$ and $L^2$}

\author{S.~V.~Astashkin}

\thanks{$^\dagger$\,The work was completed as a part of the implementation of the development program of the Scientific and Educational Mathematical Center Volga Federal District, agreement no. 075-02-2022-878.}

\maketitle

\begin{abstract} A subspace $H$ of a rearrangement invariant space $X$ on $[0,1]$ is strongly embedded in $X$ if, in $H$, convergence in $X$-norm is equivalent to convergence in measure. We obtain necessary and sufficient conditions on an Orlicz function $M$, under which the unit ball of an arbitrary strongly embedded subspace in the Orlicz space $L_M$ has equi-absolutely continuous norms in $L_M$. 
 
\end{abstract}

Primary classification: 46E30

Secondary classification(s): 46B03, 46B09, 46A45



\def\hj{{\mathbb R}}

\def\fg{{\mathbb N}}

\vskip 0.4cm

\section{Introduction}
\label{Intro}

The classical Khintchine inequality (see e.g. \cite[Chapter V, Theorem~8.4]{Z1})
asserts that, for every $0<p<\infty$, there exist
constants $A_p, B_p > 0$ such that for any sequence of real numbers
$\{c_k\}_{k=1}^\infty$ we have 
\begin{equation}\label{ineq 1}
A_p \|(c_k)\|_{\ell^2}\leq \Big\|\sum_{k=1}^\infty c_k r_k\Big\|_{p} \leq B_p \|(c_k)\|_{\ell^2},
\end{equation}
where $\|f\|_p:=\|f\|_{L_p[0,1]}$, $\|(c_k)\|_{\ell^2}:=\left(\sum_{k=1}^\infty c_k^2 \right)^{1/2}$ and
$r_k$'s  are the Rademacher functions, that is, $r_k(t) = {\rm sign}
(\sin 2^k \pi t), ~ k \in {\Bbb N}, t \in [0, 1].$
Thus, for every $0<p<\infty$ the sequence $\{r_k\}$ is equivalent in $L^p$ to the standard basis of $\ell^2$.
The following notion reflects general phenomena exemplified above by the Rademacher system. A closed subspace $H$ of the space $L^p=L^p[0,1],$ 
$1\le p<\infty,$ is said to be a {\it $\Lambda(p)$-space} if, in $H$, convergence in $L^p$-norm is equivalent to convergence in measure. Equivalently, for each (for some) $0<q<p$ there exists  
a constant $C_q>0$ such that  
\begin{equation}\label{ineq 1 extra}
 \|f\|_p\le C_q\|f\|_q\;\;\mbox{for all}\;\;f\in H
\end{equation}
(see e.g. \cite[Proposition~6.4.5]{AK}).

Thus, inequality \eqref{ineq 1} shows that
the closed linear span $[r_k]$ (say, in $L^1$) is a $\Lambda(p)$-space for every $1\le p<\infty.$


It is worth emphasizing that originally the concept of a $\Lambda(p)$-space was introduced by Rudin
\cite{Rud}, in the special setting of Fourier analysis on the circle group
$[0,2\pi).$ Namely, let $0<p<\infty.$ A set $E\subset \mathbb{Z}$ is called a {\it $\Lambda(p)$-set}
if for some $0<q<p$ there exists a constant $C_q>0$ such that inequality \eqref{ineq 1 extra}
holds for every polynomial $f$ with spectrum 
(i.e., support of Fourier transform) in $E$. Equivalently, the subspace
$H=L_E$ generated by a set of exponentials $\{e^{2\pi inx},$ $n\in E\}$ is a $\Lambda(p)$-space. 
In \cite{Rud}, for all integers $n>1$, Rudin constructed 
$\Lambda(2n)$-sets that are not $\Lambda(q)$-sets for every $q>2n.$ In 1989, Bourgain proved
a deep result extending Rudin's theorem to all $p>2$ \cite{bour}. In contrast to that, in 1974, Bachelis and Ebenstein \cite{BachEb} showed that in the case when $p\in (1,2)$ every $\Lambda(p)$-set
is a $\Lambda(q)$-set for some $q>p.$ Moreover, in \cite[Theorem~13]{Ros} Rosenthal proved that, for every $1\le p<2$ and each closed linear subspace $H$ of $L^p$ the following conditions are equivalent:

(a) $H$ does not contain any subspace isomorphic to $\ell^p$;

(b) $H$ is a $\Lambda(p)$-space;

(c) the unit ball $B_H$ of $H$ has equi-absolutely continuous norms in $L^p$\footnote{Note that implication $(c)\Longrightarrow (b)$ is trivial (see also Lemma  \ref{prop 1} below).}.

\vskip0.3cm

According to Definition~6.4.4 from \cite{AK}, we expand the notion of a $\Lambda(p)$-subset, saying that  a subspace $H$ of a rearrangement invariant function space $X$ on $[0,1]$ (for any undefined terminology and notation see in Section \ref{prel} below) is {\it strongly embedded} in $X$ if on $H$ convergence in the $X$-norm is equivalent to convergence in measure. The main purpose of the paper is extending the above Rosenthal theorem to the setting of Orlicz function spaces.
More precisely, we look for necessary and sufficient conditions on an Orlicz function $M$, under which the unit ball $B_H$ of an arbitrary strongly embedded subspace $H$ of the Orlicz space $L_M$ has equi-absolutely continuous norms in $L_M$.

We note, first, that the condition $1\le p<2$ in  Rosenthal's theorem is essential, and an analogue of this theorem may only hold for spaces lying between $L^1$ and $L^2$. Therefore, we will always assume that, for the Matuszewska-Orlicz indices $\alpha_M^\infty$ and $\beta_M^\infty$ of an Orlicz function $M$ which generates the given Orlicz space $L_M$, the inequality $1<\alpha_M^\infty\le \beta_M^\infty<2$ holds.  Then, the main result (Theorem \ref{theorem-main}) shows that an extension of Rosenthal's theorem holds for an Orlicz space $L_M$ whenever $1/\psi^{-1}\not\in L_M$ for every $\psi\in C_{M}^{\infty}$.
Note that condition (a) in this more general setting takes the form: a subspace $H$ does not contain any infinite-dimensional subspace isomorphic to a subspace spanned in $L_M$ by pairwise disjoint functions. On the other hand, Rosenthal's theorem does not extend to an Orlicz space $L_M$ if there is a function $\psi\in C_{M}^{\infty}$ satisfying some submultiplicativity condition and such that $1/\psi^{-1}\in L_M$ (see Theorem \ref{theorem-main2}).

In the case of Orlicz functions regularly varying at $\infty$ we have the following simple criterion: if $M$ is such an Orlicz function of order $p$, then the unit ball $B_H$ of an arbitrary strongly embedded subspace $H$ of the Orlicz space $L_M$ has equi-absolutely continuous norms in $L_M$ if and only if $t^{-1/p}\not\in L_M$ (see Corollary \ref{reg var}). Moreover, the same criterion still holds for an arbitrary Orlicz function $M$ such that $1<\alpha_M^\infty\le \beta_M^\infty<2$  if we confine ourselves to considering only subspaces of an Orlicz space that are isomorphic to Orlicz sequence spaces (see Theorem \ref{theorem-main5}.

The proof of Rosenthal's theorem in \cite{Ros} (see also \cite[Theorem~7.2.6]{AK}), based on applying the theory of absolute summing operators, strongly influenced the development of factorization theory (see, for instance, \cite[Chapter~7]{AK}). 
Unlike this, our approach does not use the technique of absolutely summing operators and seems to be simpler. It arose as a result of comparing the simple fact that every sequence of pairwise disjoint normed functions generates in $L^p$ the space $\ell^p$ with implication $(b)\Longrightarrow (a)$ of Rosenthal's theorem that says that $L^p$, $1\le p<2$, does not contain any $\Lambda(p)$-subspace, which is isomorphic to $\ell^p$.
This comparison suggests that an extension of Rosenthal's theorem is valid for an Orlicz space $L_M$ whenever it does not contain strongly embedded subspaces isomorphic to Orlicz sequence spaces generated in $L_M$ by pairwise disjoint functions (see Proposition \ref{prop1}) . Another important ingredient of the proof is the author's results obtained in the paper \cite{A-16}, which, in particular, imply that the space $L_M$, with $1<\alpha_M^\infty\le \beta_M^\infty<2 $, contains the function $1/\psi^{-1}$ provided that $L_M$ has a strongly embedded subspace isomorphic to the Orlicz sequence space $\ell_\psi$.

\section{Preliminaries}
\label{prel}

In what follows given two positive functions (quasinorms) $F_1$ and
$F_2$ are said to be  equivalent (we write $F_1\asymp F_2$) if there
exists a positive constant $C$ that does not depend on the arguments of $F_1$ and $F_2$ such that $C^{-1}F_1\leq F_2\leq
CF_1$. Sometimes, we say that these functions are equivalent for
large (or small) values of the argument, meaning that the
preceding inequalities hold only for its specified values.
Throughout we use the letters $C$ and $c$ to denote positive constants
whose values may change at different occurrences.

\subsection{Rearrangement invariant spaces}
\label{prel1}

Detailed exposition of theory of rearrangement invariant spaces see
in \cite{KPS,LT,BSh}.

A Banach space $X$ of real-valued Lebesgue--measurable functions on
the measure space $(I,m)$, where $I=[0,1]$ or $(0,\infty)$ and $m$ is the Lebesgue measure, is called {\it rearrangement invariant} (in brief, r.i.) (or {\it symmetric}) if from the conditions $y \in X$ and
$x^*(t)\le y^*(t)$ almost everywhere (a.e.) on $I$ it follows that $x\in X$ and ${\|x\|}_X \le {\|y\|}_X.$ Here and below, $x^*(t)$ is the right-continuous nonincreasing {\it rearrangement} of $|x(s)|$, i.e., 
$$
x^{*}(t):=\inf \{ \tau\ge 0:\,n_x(\tau)\le t \},\;\;0<t<m(I),$$
where
$$
n_x(\tau):=m\{s\in I:\,|x(s)|>\tau\},\;\;\tau>0.$$
If $x$ and $y$ are {\it equimeasurable} functions, that is, $n_x(\tau)=n_y(\tau)$ for all $\tau>0$, and $y\in X$, then $x\in X$ and ${\|x\|}_X ={\|y\|}_X.$ In particular, any measurable function $x(t)$ and its rearrangement $x^*(t)$ are  equimeasurable.

Every r.i.\ space $X$ on $[0,1]$ (resp. $(0,\infty)$) satisfies  the embeddings $L^\infty[0,1]{\subseteq} X{\subseteq} L^1[0,1]$
(resp. 
$$
(L^1\cap L^\infty)(0,\infty){\subseteq} X {\subseteq} (L^1+L^\infty)(0,\infty)).
$$

The {\it fundamental function} of $X$ is defined by $\phi_X(t):=\|\chi_A\|_X$, where $\chi_A$ is the characteristic function of a measurable set $A\subset I$ with $m(A)=t$.  The function $\phi_X$ is {\it quasi-concave} (i.e., $\phi_X(0)=0$, $\phi_X$ increases and $\phi_X(t)/t$ decreases).

For any $\tau>0$, the {\it dilation operator} ${\sigma}_\tau x(t):=x(t/\tau)\chi_{(0,\min\{1,\tau\})}(t)$, $0\le t\le 1$, is bounded in any r.i. space $X$ on $[0,1]$ and $\|{\sigma}_\tau\|_{X\to X}\le \max(1,\tau)$; see e.g. \cite[Theorem 2.4.4]{KPS}.

If $X$ is a r.i. space on the interval $[0,1]$, then {\it the
K\" othe dual} space $X'$ consists of all measurable functions $y$ such that
$$
\|y\|_{X'}\,\,=\,\,\sup\,\biggl\{\int_{0}^1{x(t)y(t)\,dt}:\;\;
\|x\|_{X}\,\leq{1}\biggr\}\,<\,\infty.
$$
The space $X'$ is r.i. as well; it is embedded into the dual space
$X^*$ of $X$ isometrically, and $X'=X^*$ if and only if $X$ is
separable. A r.i. space $X$ is said {\it to have the Fatou
property} if the conditions $x_n\in X$, $n=1,2,\dots,$
$\sup_{n=1,2,\dots}\|x_n\|_X<\infty$, and $x_n\to{x}$ a.e. imply
that $x\in X$ and $||x||_X\le \liminf_{n\to\infty}{||x_n||_X}.$  A r.i. space $X$ has the Fatou property if and only if the natural embedding of $X$
into its second K\" othe dual $X''$ is an isometric surjection. 

Quite similarly we can define also r.i.\ sequence spaces. In particular, the {\it fundamental function} of a  r.i.\ sequence space $X$ is defined by $\phi_{X}(n):=\|\sum_{k=1}^n e_k\|_{X}$, $n=1,2,\dots$. In what follows, $e_k$ are the canonical unit vectors, i.e., $e_k=(e_k^i)_{i=1}^\infty$, $e_k^i=0$ for $i\ne k$ and $e_k^k=1$, $k,i=1,2,\dots$.

The family of r.i.\ spaces includes many classical spaces 
appearing in analysis, in particular, $L^p$-spaces, Orlicz, Marcinkiwicz, Lorentz spaces and many others.

Let $1\le p<\infty$, and let $\varphi$ be an increasing concave function on $I$ such that $\varphi(0)=0$.
The {\it Lorentz space} $\Lambda_p(\varphi):=\Lambda_p(\varphi)[0,1]$ consists of all functions $x(t)$ measurable on $[0,1]$ and satisfying the condition:
\begin{equation}
\label{eqLor} 
\|x\|_{\Lambda_p(\varphi)}:=\Big(\int_0^1 x^{*}(t)^{p} 
d\varphi(t) \Big)^{1/p}<\infty
\end{equation} 
(see \cite{KPS}, \cite[p. 121]{LT}). In particular, $(\Lambda_1(\varphi))'=M(\varphi)$ \cite[Theorem~II.5.2]{KPS}, where $M(\varphi)$ is the {\it Marcinkiewicz space} equipped with the norm
$$
{\|x\|}_{M(\varphi)}:=\sup\limits_{0<t\le
1}\,\frac{1}{\varphi(t)}\int_{0}^{t}x^{\ast}(s)\,ds.$$

For every $1\le p<\infty$ and $\varphi$, $\Lambda_p(\varphi)$ is a separable r.i. space with the Fatou property, while $M(\varphi)$ is a non-separable r.i. space (if $\lim_{t\to 0}\varphi(t)/t=\infty$) with the Fatou property, and $\phi_{\Lambda_p(\varphi)}(t)=\varphi(t)^{1/p}$, $\phi_{M(\varphi)}(t)=t/\varphi(t)$, $0<t\le 1$.

Some preliminaries from the theory of Orlicz spaces, which are the main subject of this paper, we collect in the next subsection.

\subsection{Orlicz function and sequence spaces}
\label{prel2}

Most important examples of r.i. spaces are the $L_p$-spaces, $1 \leq p \leq\infty$, and their natural generalization, the Orlicz spaces (for their detailed theory we refer to the monographs \cite{KR,RR,Mal}).

Let $M$ be an Orlicz function, that is, an increasing convex continuous function on $[0, \infty)$ such that $M(0) = 0$ and $\lim_{t\to\infty}M(t)=\infty$. In what follows, we assume also that $M(1)=1$. Denote by $L_{M}:=L_M(I)$ the {\it Orlicz space} endowed with the Luxemburg--Nakano norm
$$
\| f \|_{L_{M}}: = \inf \left\{\lambda > 0 \colon \int_I M\Big(\frac{|f(t)|}{\lambda}\Big) \, dt \leq 1 \right\}.
$$
In particular, if $M(u)=u^p$, $1\le p<\infty$, we obtain $L_p$.

Note that the definition of an Orlicz function space $L_M[0,1]$ depends (up to equivalence of norms) only on the behaviour of the function $M(t)$ for large values of argument  $t$. The fundamental function $\phi_{L_M}(u)=1/M^{-1}(1/u)$, $0<u\le 1$, where $M^{-1}$ is the inverse function.

If $M$ is an Orlicz function, then the {\it Young conjugate} function $\tilde{M}$ is defined by
$$
\tilde{M}(u):=\sup_{t>0}(ut-M(t)),\;\;u>0.$$
Note that $\tilde{M}$ is also an Orlicz function and the Young conjugate for $\tilde{M}$ is $M$.

Every Orlicz space $L_M(I)$ has the Fatou property; $L_M[0,1]$ (resp. $L_M(0,\infty)$) is separable if and only if the function $M$ satisfies the {\it $\Delta_2^\infty$-condition} (resp. {\it $\Delta_2$-condition}), i.e., $\sup_{u\ge 1}{M(2u)}/{M(u)}<\infty$ (resp.  $\sup_{u>0}{M(2u)}/{M(u)}<\infty$). In this case we have  $L_M(I)^*=L_M(I)'=L_{\tilde{M}}(I)$.

Let $M$ be an Orlicz function, $M\in \Delta_{2}^{\infty}$. Define the following subsets of the space $C[0, 1]$:

$$
E_{M, A}^{\infty}: = \overline{\big\{ N(x): = \frac{M(xy)}{M(y)} \ : y > A \big\}}\;(A>0), 
E_{M}^{\infty}: = \bigcap_{A > 0}E_{M, A}^{\infty}, \ \ C_{M}^{\infty}: = \overline{{\rm conv}\, E_{M}^{\infty}},
$$
where ${\rm conv}\,U$ is the convex hull of a set $U$ and the closure is taken in the space $C[0,1].$ All these sets are non-void, compact in $C[0,1]$ and consist of Orlicz functions \cite[Lemma~4.a.6]{LT77}. It is well known that the sets $E_{M}^{\infty}$ and $C_{M}^{\infty}$ rather fully determine the structure of subspaces spanned by sequences of disjoint functions in the Orlicz space $L_M[0,1]$. In particular, $L_M[0,1]$ contains a sequence of pairwise disjoint functions $\{f_n\}_{n=1}^\infty$ such that their span $[f_n]$ is isomorphic to an Orlicz space $\ell_\psi$ if and only if $\psi\in C_{M}^{\infty}$  \cite[Proposition~4]{LTIII}.

Let $M$ be an Orlicz function and $1 \leq p < \infty.$ 
Then, $t^{p} \in C_{M}^{\infty}$ if and only if $p \in [\alpha_{M}^{\infty}, \beta_{M}^{\infty}],$ where $\alpha_{M}^{\infty}$ and $\beta_{M}^{\infty}$ are the {\it lower and upper Matuszewska-Orlicz indices for large arguments or at infinity} defined by
$$
\alpha_{M}^{\infty}: = \sup \big\{ p : \sup_{t, s \geq 1} \frac{M(t)s^{p}}{M(ts)} < \infty \big\}, \ \ \ \beta_{M}^{\infty}: = \inf \big\{ p : \inf_{t, s \geq 1} \frac{M(t)s^{p}}{M(ts)} > 0 \big\}
$$
(see \cite{LTIII} or \cite[Proposition~5.3]{KamRy}). It is easy to check that $1 \leq \alpha_{M}^{\infty} \leq \beta_{M}^{\infty} \leq \infty$. 

Note that the Matuszewska-Orlicz indices $\alpha_{M}^{\infty}$ and $\beta_{M}^{\infty}$ are a partial case of the so-called Boyd indices, which can be defined for every r.i.\ space on $[0,1]$ or $(0,\infty)$ (see, e.g., \cite[Definition~2.b.1]{LT} or \cite[\S\,II.4, p.~134]{KPS}).


\vskip0.2cm

Similarly, we can define an {\it Orlicz sequence space}. Specifically, the space $\ell_{\psi}$, where $\psi$ is an Orlicz function, consists of all sequences $(a_{k})_{k=1}^{\infty}$ such that 
$$
\| (a_{k})_{k=1}^{\infty}\|_{\ell_{\psi}} := \inf\left\{u>0: \sum_{k=1}^{\infty} \psi \Big( \frac{|a_{k}|}{u} \Big)\leq 1\right\}<\infty.
$$
Clearly, if $\psi(t)=t^p$, $p\ge 1$, then $l_\psi=\ell^p$ isometrically.

The fundamental function of an Orlicz sequence space $\ell_{\psi}$ can be calculated by the formula: $\phi_{\ell_{\psi}}(n)=\frac{1}{\psi^{-1}(1/n)}$, $n=1,2,\dots$ Furthermore, an Orlicz sequence space $\ell_{\psi}$ is separable if and only if  $\psi$ satisfies the $\Delta_{2}^{0}$-condition ($\psi\in \Delta_{2}^{0}$), that is, 
$$
\sup_{0<u\le 1}{\psi(2u)}/{\psi(u)}<\infty.$$
In this case we have  $\ell_{\psi}^*=\ell_{\psi}'=\ell_{\tilde{\psi}}$, with the Young conjugate function  $\tilde{\psi}$ for $\psi$.

Observe that the definition of an Orlicz sequence space $\ell_{\psi}$ depends (up to equivalence of norms) only on the behaviour of the function $\psi$ near zero. More precisely, if $\varphi,\psi \in \Delta_{2}^{0}$, the following conditions are equivalent: (1) $\ell_{\psi}=\ell_{\varphi}$ (with equivalence of norms); 2) the unit vector bases of the spaces $\ell_{\psi}$ and $\ell_{\varphi}$ are equivalent; 3) there are $C > 0$ and $t_{0} > 0$ such that for all $0 \leq t \leq t_{0}$ it holds
$$
C^{-1}\psi(t) \leq \varphi(t) \leq C\psi(t)
$$ 
(cf. \cite[Proposition~4.a.5]{LT77} or \cite[Theorem~3.4]{Mal}). In the case when $\psi$ is a {\it degenerate} Orlicz function, i.e., there is $t> 0$ such that $\psi(t)=0$, we have $\ell_{\psi}=\ell_\infty$ (with equivalence of norms).

One can easily see (cf. \cite[Proposition~4.a.2]{LT77}) that the canonical unit vectors $e_n$, $n=1,2,\dots$, form a symmetric basis of an Orlicz sequence space $\ell_{\psi}$ provided if $\psi\in \Delta_{2}^{0}$. Recall that a basis $\{x_n\}_{n=1}^\infty$ of a Banach space $X$ is called {\it symmetric} if there exists $C>0$ such that for an arbitrary permutation $\pi:\,\mathbb{N}\to\mathbb{N}$ and any $a_n\in\mathbb{R}$ we have
$$
C^{-1}\Big\|\sum_{n=1}^{\infty}a_nx_n\Big\|_X\le \Big\|\sum_{n=1}^{\infty}a_nx_{\pi(n)}\Big\|_X\le C\Big\|\sum_{n=1}^{\infty}a_nx_n\Big\|_X.$$

Given an Orlicz function $\psi$ we define the {\it lower and upper Matuszewska-Orlicz indices for small arguments or at zero} $\alpha_{\psi}^{0}$ and $\beta_{\psi}^{0}$ as follows:
$$
\alpha_{\psi}^{0}: = \sup \big\{ p : \sup_{0<t, s \leq 1} \frac{\psi(st)}{s^{p}\psi(t)} < \infty \big\}, \ \ \ \beta_{\psi}^{0}: = \inf \big\{ p : \inf_{0<t, s \leq 1} \frac{\psi(st)}{s^{p}\psi(t)} > 0 \big\}.
$$
Similarly as for the indices at infinity, we have $1 \leq \alpha_{\psi}^{\infty} \leq \beta_{\psi}^{\infty} \leq \infty$ (see e.g. \cite[Chapter~4]{LT77}). Moreover, the space $\ell^p$, or $c_0$ if $p=\infty$, is isomorphic to a subspace of $\ell_\psi$ if and only if $p\in [\alpha_{\psi}^{0},\beta_{\psi}^{0}]$ 
\cite[Theorem~4.a.9]{LT77}.

 Let $1\leq p<q<\infty$. We say that an Orlicz function $M$ is $p$-{\it convex} if the map $t \mapsto M(t^{1/p})$ is convex
and $q$-{\it concave} if the map $t \mapsto M(t^{1/q})$ is concave. It
is easy to verify that an Orlicz space $L_M[0,1]$ is $p$-convex
($q$-concave) if and only if $M(t)$ is equivalent to a $p$-convex
($q$-concave) Orlicz function for large values of $t$. Similarly,
an Orlicz sequence space $l_M$ is $p$-convex ($q$-concave) if and
only if  $M(t)$ is equivalent to a $p$-convex ($q$-concave) Orlicz
function for small values of $t$. Recall that a Banach lattice $X$
is said to be {\it $p$--convex} ($1 \leq p \le\infty$),
respectively, {\it $q$--concave} ($1 \leq q \le \infty$) if
$$
\Big\|\Big(\sum_{k=1}^n |x_k|^p\Big)^{1/p}\Big\|_X \leq C \Big(\sum_{k=1}^n \|x_k\|_X^p\Big)^{1/p},
$$
respectively,
$$
\Big(\sum_{k=1}^n\|x_k\|_X^q\Big)^{1/q} \leq C \Big\| \Big(\sum_{k=1}^n |x_k|^q\Big)^{1/q}\Big\|_X
$$
(with a natural modification in the case $p=\infty$ or $q=\infty$)
for some constant $C>0$ and every choice of vectors $x_{1}, x_{2}, \dots, x_{n}$ in $X$. Of course, every Banach lattice is $1$-convex and $\infty$-concave with constant $1$. Also, the space $L^p$, $1\\le p\le\infty$, is $p$-convex and $p$-concave with constant $1$.

An Orlicz function $M$ such that $M\in \Delta_{2}^{\infty}$ is said to be {\it regularly varying at $\infty$} (in sense of Karamata) if the limit $\lim_{t\to\infty}M(tu)/M(t)$ exists for all $u>0$ (in fact, it suffices that the limit exists when $0<u\le 1$). Then, there is a $p$, $1\le p<\infty$, such that $\lim_{t\to\infty}M(tu)/M(t)=u^p$; in this case $M$ is called {\it regularly varying of order $p$}.




\subsection{The characteristic $\eta_X(K)$ and subsets of r.i.\ function spaces with equi-absolutely continuous norms}
\label{prel3}

Let $X$ be a r.i.\ space on $[0,1]$. For any set $K\subset X$, let
$$
\eta_X(K):=\lim_{t\to 0}\sup_{x\in K,x\ne 0}\frac{\|x^*\chi_{[0,t]}\|_X}{\|x\|_X}.$$
The characteristic $\eta_X(K)$ was explicitly introduced
by Tokarev in \cite{Tok}, although for $X=L_1$ it had arisen much earlier in the classical paper by Kadec and Pe{\l}czy{\'n}ski  \cite{KadPel}. Later on, this characteristic was studied in \cite{NST} and \cite{NST2}. 

Let $X$ be a r.i.\ space. Clearly, one always has $0\le \eta_X(K)\le 1$. It can be also easily seen that the inequality $\eta_X(K)<1$ implies the equivalence of the $X$- and $L^1$-norms on the set $K$. Moreover, if additionally $X\ne L^1$ and $K$ is a closed linear subspace of $X$, we get that $K$ is strongly embedded into $X$ (see the proof of Proposition~6.4.5 in \cite{AK}). If $X$ is separable, the latter condition, in view of the generalization of the well-known Kadec-Pe{\l}czy{\'n}ski alternative (see \cite[Theorem~4.1]{FJT}, \cite[Proposition~1.c.8 and its proof]{LT}, or \cite[Lemma~5.2.1]{AK} for $X=L^1$), is equivalent to the fact that $K$ does not contain almost disjoint sequences (i.e., there is no sequence $\{x_n\}_{n=1}^\infty\subset K$, $\|x_n\|_X=1$, such that for some pairwise disjoint functions $u_n\in X$, $n=1,2,\dots$, it holds $\|u_n-x_n\|_X\to 0$ as $n\to\infty$). 

In the case when $\eta_X(K)=0$ we can say about the set $K\subset X$ much more.


Let $X$ be a r.i.\ space on $[0,1].$ We shall say that a set
$K\subset X$ has {\it equi-absolutely continuous norms} in $X$ if
$$
\lim_{\delta\to 0}\sup_{m(E)<\delta}\sup_{x\in K}\|x\chi_{E}\|_X=0.$$

The proofs of the following lemmas are immediate and hence they are skipped. 
\begin{lemma}\label{lemma 2}
For arbitrary r.i. space $X$ on $[0,1]$ a bounded set $K\subset X$ has equi-absolutely continuous norms in $X$ if and only if
$\eta_X(K)=0$.

\end{lemma}

Recall (see Section \ref{Intro}) that a subspace $H$ of a r.i.\ space $X$ on $[0,1]$ is  strongly embedded in $X$ if, in $H$, convergence in the $X$-norm is equivalent to convergence in measure.
Throughout by $B_H$ we denote the unit ball of a closed subspace $H$ of a r.i.\ space $X$, that is, $B_H:=\{x\in H:\,\|x\|_X\le 1\}$.

\begin{lemma}\label{prop 1}
Let $H$ be a closed subspace of a r.i.\ space $X$ on $[0,1].$ If the set $B_H$ has equi-absolutely continuous norms in $X,$ then $H$ is strongly embedded into $X$.
\end{lemma}

Following to \cite{NST2}, we will say that a r.i.\ space $X$ on $[0,1]$ is {\it binary} if the characteristic $\eta_X(H)$ takes on closed linear subspaces $H$ of $X$ only two values: $0$ and $1$. From the above results it follows that a r.i.\ space $X$ is binary if the unit ball of each closed subspace $H\subset X$ that is strongly embedded in $X$ has equi-absolutely continuous norms in $X$. The converse is not true in general, because there are r.i.\ spaces $X$ that contain strongly embedded subspaces $H$ with $\eta_X(H)=1$ (see e.g.  \cite[Theorem~3]{AS-20}). 

The above discussion combined with Rosenthal's theorem \cite[Theorem~13]{Ros} (see also Section \ref{Intro}) and Lemma \ref{lemma 2} shows that the space $L^p$ is binary whenever  $1\le p<2$. In \cite{NST2} (see also \cite{NST}) this result has been extended to the class of Lorentz spaces $\Lambda_p(\varphi)$ such that $1\le p<2$. Moreover, it was proved (see \cite[Theorem~4]{NST2}) that the condition $\eta_{\Lambda_p(\varphi)}(H)=1$ implies that the subspace $H$ contains an almost disjoint sequence of functions equivalent to the unit vector basis in $\ell^p$.

\section{Auxiliary results}

\subsection{A description of subspaces of Orlicz spaces spanned by mean zero independent identically distributed functions.\\}
\label{aux1}

Given a sequence $\{f_k\}_{k=1}^\infty$ of measurable functions on $[0,1]$ we define the {\it disjoint sum} $\sum_{k=1}^\infty \oplus f_k$ to be any function ${\bf f}$ on $(0,\infty)$ with $n_{{\bf f}}(\tau)=\sum_{k=1}^\infty n_{f_k}(\tau)$ for all $\tau>0$. In particular, we can take for ${\bf f}$ the function $\sum_{k=1}^\infty f_k(t-k+1)\chi_{[k-1,k)}(t)$, $t>0$.

Let $M$ be an Orlicz function, $M\in \Delta_2^\infty$, and let $L_M=L_M[0,1]$ be the Orlicz space. Suppose that $\{f_k\}_{k=1}^\infty\subset L_M$ is a sequence of independent functions such that $\int_0^1 f_k(t)\,dt=0$, $k=1,2,\dots$.  Then, by the well-known Johnson-Schechtman result \cite[Theorem~1]{JS}, we have
\begin{equation}
\label{JoSh}
\Big\|\sum_{k=1}^\infty f_k\Big\|_{L_M}\asymp \Big\|\Big(\sum_{k=1}^\infty  \oplus f_k\Big)^*\chi_{[0,1]}\Big\|_{L_M}+\Big\|\Big(\sum_{k=1}^\infty  \oplus f_k\Big)^*\chi_{[1,\infty)}\Big\|_{L^2[1,\infty)}.
\end{equation}
We can rewrite this equivalently as follows:
\begin{equation}
\label{dis}
\Big\|\sum_{k=1}^\infty f_k\Big\|_{L_M}\asymp \Big\|\sum_{k=1}^\infty  \oplus f_k\Big\|_{L_\theta},
\end{equation}
where $L_\theta$ is the Orlicz space on $(0,\infty)$, $\theta(u)=u^2$ if $0<u\le 1$, and  $\theta(u)=M(u)$ if $u\ge 1$.

Let us assume additionally that there is a function $f\in L_M$ such that each of the functions $f_k$, $k=1,2,\dots$, is equimeasurable with $f$ . Then, substituting in equivalence \eqref{dis} scalar multiples $a_kf_k$, $a_k\in\mathbb{R}$, $k=1,2,\dots$, instead of $f_k$, by the definition of the norm in $L_M$, we obtain 
\begin{eqnarray}
\Big\|\sum_{k=1}^\infty a_kf_k\Big\|_{L_M}&\asymp&\inf\{\lambda >0:\,\int_{0}^{\infty}\theta\Big(\frac1{\lambda}\sum_{k=1}^\infty  \oplus |a_kf_k(t)|\Big)\,dt\le 1\}\nonumber\\
&=&\inf\{\lambda >0:\,\sum_{k=1}^\infty\int_{0}^{1}\theta\Big(\frac{|a_k||f(t)|}{\lambda}\Big)\,dt\le 1\}=\|(a_k)\|_{{\ell_\psi}},
\label{Luxem}
\end{eqnarray}
where 
\begin{equation}\label{psi}
\psi(u):=\int_0^1\theta(u|f(t)|)dt,\;\;u>0.
\end{equation}

Thus, $\{f_k\}_{_{k=1}}^\infty$ is equivalent in $L_M$ to the unit vector basis in the Orlicz sequence space ${\ell_\psi}$, where the function $\psi$ is defined in \eqref{psi}.

Observe that $\theta$ need not to be an Orlicz function. However, the function $\theta(t)/t$ does not decrease and, as $M\in \Delta_2^\infty$, we have $\theta\in\Delta_2$. Hence, one can easily check that $\theta$ is equivalent on $(0,\infty)$ to the Orlicz function $\tilde{\theta}(t):=\int_0^t\frac{\theta(u)}{u}\,du.$ In turn, this and \eqref{psi} imply that $\psi$ is equivalent on $(0,\infty)$ to some Orlicz function as well.

\subsection{Matuszewska-Orlicz indices and strongly embedded subspaces of Orlicz spaces.\\}
\label{aux2}

The following result is a consequence of \cite[Lemma 20]{MSS} (see also  \cite[Lemma~5]{AS-14} for the full proof) and the definition of Matuszewska-Orlicz indices.

\begin{lemma}
\label{Lemma 20}
Let $1\leq p<\infty$ and let $\psi$ be an Orlicz function on $[0,\infty)$. Then,

(i) $\psi$ is equivalent to a $p$-convex (resp. $p$-concave) Orlicz function for small values of the argument $\Longleftrightarrow$ $\psi(st)\le C s^{p}\psi(t)$ (resp. $s^p\psi(t)\le C \psi(st)$) for some $C>0$ and all $0<t,s\leq 1$;

(ii) $\psi$ is equivalent to a $(p+\varepsilon)$-convex (resp. $(p-\varepsilon)$-concave) Orlicz function for small values of the argument
and some $\varepsilon>0$ $\Longleftrightarrow$ $\alpha_\psi^0>p$ (resp. $\beta_\psi^0<p$).

\end{lemma}


\begin{lemma}
\label{lemma1}
Let $M$ and $\psi$ be Orlicz functions such that $\psi \in C_{M}^{\infty}$. Then, $\alpha_M^\infty\le\alpha_\psi^0 \le \beta_\psi^0\le\beta_M^\infty$.
\end{lemma}
\begin{proof}
We show only that $\beta_\psi^0\le\beta_M^\infty$ because the proof of the inequality $\alpha_M^\infty\le \alpha_\psi^0$ is quite similar.

Assume first that $L\in E_M^\infty$, i.e.,
\begin{equation} 
 \label{eq1}
L(s)=\lim_{n\to\infty}\frac{M(st_n)}{M(t_n)},\;\;0<s\le 1, 
\end{equation}
where $t_n\uparrow \infty$ as $n\to\infty$.

Since 
$$
\beta_M^\infty=\inf\Big\{p>0:\,\inf_{x,y\ge 1}\frac{y^pM(x)}{M(xy)}>0\Big\},$$
then for any $\varepsilon>0$ there is $h_\varepsilon>0$ such that for $p:= \beta_M^\infty$ and all $x,y\ge 1$
\begin{equation} 
 \label{eq2}
\frac{M(x)y^{p+\varepsilon}}{M(xy)}>h_\varepsilon. 
\end{equation}

By  \eqref{eq1}, for all $0<u,v\le 1$ and $q>0$
\begin{equation*}
\frac{L(uv)}{v^{q}L(u)}=\frac{1}{v^{q}}\frac{\lim_{n\to\infty}\frac{M(uvt_n)}{M(t_n)}}{\lim_{n\to\infty}\frac{M(ut_n)}{M(t_n)}}=\lim_{n\to\infty}\frac{M(uvt_n)}{M(ut_n)v^q}=
\lim_{n\to\infty}\frac{M(uvt_n)v^{-q}}{M(v^{-1}uvt_n)}.
\end{equation*}
Observe that $v^{-1}\ge 1$ and $uvt_n\ge 1$ if $n\in\mathbb{N}$ is large enough. Therefore,  from \eqref{eq2} and the last inequality it follows that for $q=p+\varepsilon$ and all $0<u,v\le 1$ 
$$
\frac{L(uv)}{v^{p+\varepsilon}L(u)}>h_\varepsilon.$$

Now, if $N\in {\rm conv}\,E_{M}^{\infty}$, then $N=\sum_{k=1}^m\lambda_kL_k$, $L_k\in E_M^\infty$, $\lambda_k>0$, $k=1,2,\dots,m$, $\sum_{k=1}^m\lambda_k=1$. Hence, as above, for all $k=1,2,\dots,m$ and $0<u,v\le 1$ we have
$$
\frac{L_k(uv)}{v^{p+\varepsilon}L_k(u)}>h_\varepsilon,$$ 
which implies that 
$$
\frac{N(uv)}{v^{p+\varepsilon}N(u)}>h_\varepsilon.$$

Finally,  let $\psi\in C_M^\infty$. Then, for some sequence $\{N_i\}\subset {\rm conv}\,E_{M}^{\infty}$ we have
$$
\lim_{i\to\infty}\sup_{0<s\le 1}|\psi(s)-N_i(s)|=0.$$
Since $h_\varepsilon$ depends only on $\varepsilon$, as above, 
$$
\frac{N_i(uv)}{v^{p+\varepsilon}N_i(u)}>h_\varepsilon$$
for all $i=1,2,\dots$ and $0<u,v\le 1$. Passing to the limit as $i\to\infty$, we get 
$$
\inf_{0<u,v\le 1}\frac{\psi(uv)}{v^{p+\varepsilon}\psi(u)}\ge h_\varepsilon.$$
Hence,
$$
\beta_\psi^0:=\inf\Big\{r>0:\,\inf_{0<u,v\le 1}\frac{\psi(uv)}{v^r\psi(u)}>0\Big\}\le p+\varepsilon=\beta_M^\infty+\varepsilon.$$
Since $\varepsilon>0$ is arbitrary, the inequality $\beta_\psi^0\le\beta_M^\infty$ is proved.
\end{proof}


\begin{prop}
\label{prop1}
Let $M$ be an Orlicz function such that $1<\alpha_M^\infty\le \beta_M^\infty<2$. Assume also that at least one of the following conditions holds:

(a) for some $K_1>0$
\begin{equation} 
 \label{eq0}
\limsup_{v\to\infty}\frac{M(uv)}{M(v)}\le K_1M(u)\log(eu),\;\;u\ge 1, 
\end{equation}

(b) $t^{-1/\beta_M^\infty}\not\in L_M$  and for some $K_2>0$
\begin{equation} 
 \label{eq00}
\limsup_{v\to\infty}\frac{M(uv)}{M(v)}\le K_2u^{\beta_M^\infty},\;\;u\ge 1.
\end{equation}
Then, $1/\psi^{-1}\not\in L_M$ for every $\psi\in C_{M}^{\infty}$.

Moreover, if a strongly embedded subspace $H$ of $L_M$ is isomorphic to an Orlicz sequence space $\ell_\psi$, then $\alpha_\psi^0>\beta_M^\infty$ and hence $\psi\not\in C_{M}^{\infty}$.
\end{prop}

\begin{proof}
Suppose first that the condition (a) is satisfied and $\psi\in E_M^\infty$, i.e.,
\begin{equation} 
 \label{repr for E}
\psi(s)=\lim_{n\to\infty}\frac{M(st_n)}{M(t_n)},\;\;0<s\le 1,\end{equation}
for some sequence $t_n\ge 1$, $t_n\uparrow \infty$. Then, for each fixed $s\in (0,1]$ from  \eqref{eq0} it follows that 
$$
\frac{1}{\psi(s)}=\lim_{n\to\infty}\frac{M(s^{-1}st_n)}{M(st_n)}\le \limsup_{v\to\infty}\frac{M(s^{-1}v)}{M(v)}\le K_1M(1/s)\log(e/s).$$
Hence,
$$
M(1/s)\psi(s)\ge \frac1{K_1\log(e/s)},\;\;0<s\le 1.$$
Taking into account the definition of the set $C_M^\infty$, we infer that this inequality holds 
for every $\psi\in C_M^\infty$ and all $0<s\le 1$. Changing variables and using the inequality $\psi(t)\le t$, $0<t<1$, we get 
$$
M\Big(\frac{1}{\psi^{-1}(t)}\Big)\ge \frac{1}{K_1t\log(\frac{e}{\psi^{-1}(t)})}\ge \frac{1}{K_1t\log({e}/{t})},\;\;0<t<\psi(1)\le 1,$$
whence
$$
\int_0^1M\Big(\frac{1}{\psi^{-1}(t)}\Big)\,dt=\infty.$$
Since $M\in\Delta_2^\infty$, we conclude that $1/\psi^{-1}\not\in L_M$. 

Suppose now that the condition (b) holds. Then, if $\psi\in E_{M}^{\infty}$ is defined by \eqref{repr for E}, then from inequality \eqref{eq00} it follows
$$
\psi(s)=\lim_{n\to\infty}\frac{M(v_n)}{M(v_n/s)}\ge \left(\limsup_{v\to\infty}\frac{M(v/s)}{M(v)}\right)^{-1}\ge K_2^{-1}s^{\beta_M^\infty},\;\;0<s\le 1.$$
Clearly, this inequality is fulfilled for every $\psi\in C_{M}^{\infty}$ as well.
Hence, $1/\psi^{-1}(s)\ge \frac{1}{K_2} s^{-1/\beta_M^\infty}$, $0<s\le 1$. Since according to (b) $s^{-1/\beta_M^\infty}\not\in L_M$, we conclude again that $1/\psi^{-1}\not\in L_M$ for every $\psi\in C_{M}^{\infty}$. Thus, the first assertion of the proposition is proved.

To prove the second one, assume that $H$ is a strongly embedded subspace of $L_M$ that is isomorphic to some Orlicz sequence space $\ell_\psi$. Then, by \cite[Theorem~4.a.8]{LT77} (see also \S\,\ref{prel2}), $\ell_{\alpha_\psi^0}$ is a subspace of $\ell_\psi$, and hence $L_M$ contains a strongly embedded subspace isomorphic to $\ell_{\alpha_\psi^0}$. Since by condition the space $L_M$ is reflexive, we have $\alpha_\psi^0>1$. Moreover, $\beta_M^\infty<2$, and hence the required assertion is obvious if $\alpha_\psi^0\ge 2$. Thus, we can assume that   $1<\alpha_\psi^0<2$.
Therefore, since $L_M$ is separable and has the Fatou property, we can apply \cite[Corollary~3.3]{A-16} to deduce that $t^{-1/\alpha_\psi^0}\in L_M$. On the other hand, observe that $t^{\beta_M^\infty}\in C_{M}^{\infty}$ (see Section~\ref{prel2}). Consequently, by the first (already proved) assertion of the proposition, $t^{-1/\beta_M^\infty}\not\in L_M$. Comparing the last relations, we conclude that $\alpha_\psi^0>\beta_M^\infty$.
The fact that $\psi\not\in C_{M}^{\infty}$ follows now from Lemma \ref{lemma1}, and so the proof of the proposition is completed.

\end{proof}

\subsection{A weak version of de la Vall\'{e}e Poussin criterion.\\}
\label{aux3}


Do not finding a precise reference, we give, for the convenience of the reader, an independent proof of the following weak version of the well-known de la Vall\'{e}e Poussin criterion (see e.g. \cite{Al-94}, \cite{CFMN}, \cite{LMT}).

\begin{lemma}\label{lemma 3}
Let $M$ be an Orlicz function such that $M\in \Delta_2^\infty$ and $\tilde{M}\in \Delta_2^\infty$. Then, for any $f\in L_M$ there exists a function $N,$ equivalent to some Orlicz function and satisfying the following conditions: $N(1)=1,$ $N\in \Delta_2^\infty$, $\tilde{N}\in \Delta_2^\infty$, 
\begin{equation}\label{eq10}
\lim_{u\to\infty}\frac{N(u)}{M(u)}=\infty
\end{equation}
and
\begin{equation}\label{eq11}
\int_0^1N(|f(t)|)\,dt<\infty.
\end{equation}
Moreover, if additionally $M$ is $p$-convex for some $p<\infty$, then, along with the previous properties, $N$ is equivalent to a $p$-convex Orlicz function.
\end{lemma}
\begin{proof}
Without loss of generality, we assume that $\|f\|_{L_M}\le 1$. Since $M\in \Delta_2^\infty$, this implies that $\|g\|_{1}\le 1$, where $g:=M(|f|)$. Then, denoting $E_n(g):=\{t\in [0,1]:\,n-1<g(t)\le n\}$, $n=1,2,\dots$, we get 
\begin{equation}\label{eq11a}
\|g\|_1\asymp \int_{E_1(g)} g(t)\,dt+\sum_{n=2}^\infty n\cdot m(E_n(g))<\infty.
\end{equation}
Consequently, one can select inductively an increasing sequence of positive integers $\{n_k\}_{k=1}^\infty$ such that $n_{k+1}\ge 2n_k$, $k=1,2,\dots$, and
$$
\sum_{n=n_k}^\infty n\cdot m(E_n(g))<2^{-k},\;\;k=1,2,\dots$$

Let $a_n=1$ for $n=0,1,\dots,n_1,$ and $a_n=k$ for $n_k<n\le n_{k+1},$ $k=1,2,\dots$. Also, we define the function $R:\,[0,\infty)\to\mathbb{R}$ as follows: $R(n)=a_n,$ $n=0,1,2,\dots,$ and
$R$ is linear on each interval $[n,n+1]$, $n=0,1,\dots$. Then, $R$ is continuous, $R(u)\ge 1$ and $R(u)\uparrow\infty$ as $u\to\infty.$
Furthermore, from \eqref{eq11a} it follows
\begin{eqnarray*}
\int_0^1 g(t)R(g(t))\,dt &=& \sum_{n=1}^\infty \int_{E_n(g)} g(t)R(g(t))\,dt\le
\int_{E_1(g)} g(t)\,dt+\sum_{n=2}^\infty na_n\cdot m(E_n(g))\\
&\le& \int_{E_1(g)} g(t)\,dt+\sum_{n=2}^{n_1} n\cdot m(E_n(g))+
\sum_{k=1}^\infty k \sum_{n=n_k+1}^{n_{k+1}} n\cdot m(E_n(g))\\
&\le& C+\sum_{k=1}^\infty k2^{-k}<\infty.
\end{eqnarray*}

Let $N(u):=M(u)R(M(u))$, $u\ge 0$. Since $N$ and $N(u)/u$ increase, $N(0)=0$, then, as above, $N$ is equivalent to some Orlicz function. Also, the last estimate implies that the function $N$ satisfies conditions \eqref{eq10} and \eqref{eq11}. Moreover, taking into account that $n_{k+1}\ge 2n_k$, $k=1,2,\dots$, we get
$$
N(2u)=M(2u)R(M(2u))\le KM(u)R(KM(u))\le K^2M(u)R(M(u))=K^2N(u),\;\;u\ge 1,$$
where $K$ is the $\Delta_2^\infty$-constant of $M$. Thus, $N\in\Delta_2^\infty$.

Next, we recall that the Young conjugate function $\tilde{Q}$ for an Orlicz function $Q$ satisfies the $\Delta_2^\infty$-condition if and only if there exist $l>1$ and $u_0>0$ such that $Q(lu)\ge 2lQ(u)$ for all $u\ge u_0$ \cite[Theorem~I.4.2]{KR}. In particular, since $\tilde{M}\in \Delta_2^\infty$, we have $M(lu)\ge 2lM(u)$, $u\ge u_0$,  for some $l>1$ and $u_0>0$. Therefore, if $u\ge u_0$, we have
$$
N(lu)=M(lu)R(M(lu))\ge 2lM(u)R(2lM(u))\ge 2lM(u)R(M(u))=2lN(u),$$
whence $\tilde{N}\in\Delta_2^\infty$.

At last, we assume that the function $M$ is $p$-convex for some $p<\infty$. Then, for all $0<s<1$ and $u>0$
$$
N(su)=M(su)R(M(su))\le s^pM(u)R(s^pM(u))\le s^pM(u)R(M(u))=s^pN(u),$$
which implies that $N$ is equivalent to some $p$-convex function  \cite[Lemma~20]{MSS}.
\end{proof}

\section{Main results}

We start with the following key result showing a connection between the existence of a strongly embedded subspace of an Orlicz space $L_M$ with the unit ball having non-equi-absolutely continuous norms in $L_M$ and the structure of the set $C_M^\infty$.

\begin{prop}
\label{prop-main}
Suppose $M$ is an Orlicz function, $1<\alpha_M^\infty\le \beta_M^\infty<2$. Assume that there exists a strongly embedded subspace $H$ of the Orlicz space $L_M$ such that the unit ball $B_H$ of $H$ fails to have equi-absolutely continuous norms in $L_M$.

Then, there is a function $\psi\in C_M^\infty$ such that $1/\psi^{-1} \in L_M$.
\end{prop}
\begin{proof}
Since the unit ball $B_H$ of $H$ fails to have equi-absolutely continuous norms in $L_M$, we obtain
\begin{equation}\label{extra import1}
\lim_{\delta\to 0}\sup_{m(E)<\delta}\sup_{f\in B_H}\|f\chi_{E}\|_{L_M}:=4\alpha>0.
\end{equation}
We construct an unconditional bounded basic sequence $\{f_i\}_{i=1}^\infty\subset H$ and a sequence $\{F_i\}_{i=1}^\infty$ of pairwise disjoint subsets of $[0,1]$ which satisfy the estimate  
\begin{equation}
\label{extra import}
 \|f_i\chi_{F_i}\|_{L_M}\ge \alpha,\;\;i=1,2,\dots.
\end{equation}

To achieve this, first we define inductively auxiliary sequences of functions $\{g_n\}_{n=1}^\infty\subset H$, of numbers $\{\eta_n\}_{n=1}^\infty$ and of subsets $\{E_n'\}_{n=1}^\infty$ of $[0,1]$ as follows. Let $\eta_1=1$. In view of \eqref{extra import1}, we can find $E_1'\subset[0,1]$ and $g_1\in H$ such that $m(E_1')<\eta_1/2,$ $\|g_1\|_{L_M}=1,$ and $\|g_1\chi_{E_1'}\|_{L_M}\ge 3 \alpha$. Next, since $L_M$ is separable, if $\eta_2\in (0,\eta_1/2)$ is sufficiently
small, then $\|g_1\chi_{A}\|_{L_M}\le \alpha$ whenever $m(A)\le \eta_2.$ Moreover, again by \eqref{extra import1}, there are a set
$E_2'\subset [0,1],$ with $m(E_2')<\eta_2/2$, and a function $g_2\in H,$ with $\|g_2\|_{L_M}=1,$ satisfying $\|g_2\chi_{E_2'}\|_{L_M}\ge 3 \alpha$. 

Assume now that, for some positive integer $n\ge 2$, the functions $g_i\in H$, the numbers $\eta_i$ and the sets $E_i'$, $i=1,2,\dots, n-1,$  such that  $\|g_i\|_{L_M}=1,$
$\eta_{i}\in (0,\eta_{i-1}/2)$ $(i=2,\dots, n-1),$ $m(E_i')<\eta_i/2$, $\|g_i\chi_{E_i'}\|_{L_M}\ge 3 \alpha$, $i=1,\dots,n-1$, and $\|g_j\chi_{A}\|_{L_M}\le \alpha$ for all $j=1,\dots,i-1$ whenever $m(A)\le \eta_i$, $i=1,\dots,n-1$, are already chosen. 
Then, there exists $\eta_n\in (0,\eta_{n-1}/2)$ satisfying the conditions: $\|g_j\chi_{A}\|_{L_M}\le \alpha$ 
for all $1\le j\le n-1$ and for any set $A$ such that $m(A)\le \eta_n.$ Using hypothesis \eqref{extra import1} once again, we find a function $g_n\in H,$ $\|g_n\|_{L_M}=1,$ and a set $E_n'\subset [0,1],$ $m(E_n')<\eta_n/2$ such that $\|g_n\chi_{E_n'}\|_{L_M}\ge 3 \alpha$. 

Having the sequences $\{g_n\}_{n=1}^\infty$, $\{\eta_n\}_{n=1}^\infty$ and $\{E_n'\}_{n=1}^\infty$ with the above properties, we define the pairwise disjoint sets $E_n:=E_n'\setminus \cup_{i=n+1}^\infty E_i',$ $n=1,2,\dots.$
Since
$$
m(\bigcup_{i=n+1}^\infty E_i')\le \sum_{i=n+1}^\infty m(E_i')\le \frac12\sum_{i=n+1}^\infty \eta_i\le
\frac12\sum_{i=1}^\infty\frac{\eta_{n+1}}{2^{i-1}}=\eta_{n+1},$$
in view of the choice of $\eta_{n+1}$, we have
$$
\|g_n\chi_{E_n}\|_{L_M}\ge \|g_n\chi_{E_n'}\|_{L_M}-\|g_n\chi_{\cup_{i=n+1}^\infty E_i'}\|_{L_M}\ge
2\alpha,\;\;n=1,2,\dots$$
Hence, 
\begin{equation}\label{prop dd2}
\|g_n\chi_{E_n}\|_{L_M}\ge 2 \alpha\;\;\mbox{and}\;\;\|g_j\chi_{E_n}\|_{L_M}\le \alpha,\;\; 1\le j\le n-1\;,\; n=1,2,\dots
 \end{equation}

Thanks to the condition $1<\alpha_M^\infty\le \beta_M^\infty<2$, the space $L_M$ is reflexive. Consequently, one can select a weakly converging subsequence $\{g_{n_i}\}\subset \{g_n\}$. Then, setting $f_i:=g_{n_{i+1}}-g_{n_i},$ we see that $f_i\xrightarrow w 0$ in $L_M$.
Moreover, there is an  unconditional basis (say, the Haar system) in $L_M$ (see e.g. \cite[Theorem~2.c.6]{LT}). 
Therefore, by the Bessaga-Pe{\l}czy{\'n}ski Selection Principle (see, for instance, \cite[Proposition~1.a.12]{LT77}), passing 
 if it is necessary to some further subsequence of $\{f_i\}$ (and keeping the notation), we can assume that $\{f_i\}$ is equivalent in $L_M$ to some block basis of the above unconditional basis and hence it is an unconditional basic sequence  in this space. 
Moreover, $\|f_i\|_{L_M}\le 2$ for all $i=1,2,\dots,$ and from \eqref{prop dd2} it follows that
$$
\|f_i\chi_{E_{n_{i+1}}}\|_{L_M}\ge \|g_{n_{i+1}}\chi_{E_{n_{i+1}}}\|_{L_M}-\|g_{n_{i}}\chi_{E_{n_{i+1}}}\|_{L_M}\ge\alpha.$$
As a result, inequalities \eqref{extra import} hold for the above $f_i$ and $F_i:=E_{n_{i+1}}$, $i=1,2,\dots.$


Next, denoting $h_i:=f_i\chi_{F_i}$, $i=1,2,\dots$, by using unconditionality of the sequence $\{f_i\}$ and Khintchine's $L^1$-inequality \cite{szarek}, we get
\begin{eqnarray*}
\Big\|\sum_{k=1}^\infty c_kf_k\Big\|_{L_M} &\asymp& \int_0^1\Big\|\sum_{k=1}^\infty c_kf_kr_k(s)\Big\|_{L_M}\,ds\ge 
\frac{1}{\sqrt{2}}\Big\|\Big(\sum_{k=1}^\infty c_k^2f_k^2\Big)^{1/2}\Big\|_{L_M}\\ &\ge&\frac{1}{\sqrt{2}}\Big\|\Big(\sum_{k=1}^\infty c_k^2f_k^2\chi_{F_k}\Big)^{1/2}\Big\|_{L_M}= \frac{1}{\sqrt{2}}\Big\|\sum_{k=1}^\infty c_kh_k\Big\|_{L_M}.
\end{eqnarray*}
Hence, taking into account that the subspace $[f_i]$ is strongly embedded in $L_M$, we have
\begin{equation}\label{equiv_basis}
\Big\|\sum_{k=1}^\infty c_kf_k\Big\|_{1}\ge c\Big\|\sum_{k=1}^\infty c_kh_k\Big\|_{L_M}
\end{equation}
for some $c>0$.

Moreover, since  $\{f_n\}$ is a weakly null sequence in $L_M$, which  is separable and has the Fatou property, by a version of the  subsequence splitting property, proved in \cite[Proposition~3.2]{DSS} (see also \cite[Lemma 3.6]{ASS}), passing to some subsequence (and again preserving the notation), we obtain 
$$
f_n=u_n+v_n+w_n,\;\;n=1,2,\dots,$$
where $\{u_n\},\{v_n\},\{w_n\}$ are sequences from $L_M$ such that $u_n^*=u_1^*,$ $v_n$ are pairwise disjoint, $\lim_{n\to\infty}\|w_n\|_{L_M}=0, u_n\xrightarrow w 0,  v_n\xrightarrow w 0$ in $L_M$. 
It is clear that $v_n\xrightarrow w 0$ in $L_1$, and therefore, by disjointness, $\|v_n\|_1\rightarrow 0$. 
Hence, from the stability property of a basic sequence (see, e.g.,  \cite[Theorem~1.3.9]{AK}) and \eqref{equiv_basis} (passing again to some further subsequence) it follows that 
\begin{equation}\label{equiv_1l_M}
\Big\|\sum_{k=1}^\infty c_ku_k\Big\|_1\ge c' \Big\|\sum_{k=1}^\infty c_kh_k\Big\|_{L_M}
\end{equation}
for some $c'>0$.

Furthermore, by \cite[Theorem 4.5]{SS} (see also \cite[Proposition~2.1]{AKS}), we can select a subsequence of $\{u_n\}$
(again we keep the notation) such that $u_n=x_n+y_n,$ $n=1,2,\dots$, 
where $\{x_n\}$ is a sequence of martingale differences with respect to the standard dyadic filtration on $[0,1]$, $x_n\xrightarrow w 0$ in $L_1$ and $\|y_n\|_1\rightarrow 0$. Then, by \cite[Lemma 3.5]{ASS} 
(for general results on comparison of norms of sums of martingale differences and their disjoint copies in r.i.\ spaces see  \cite{ASW1}), we obtain
$$\Big\|\sum_{i=1}^nx_i\Big\|_1\leq C_1\Big\|\sum_{i=1}^n \oplus{x_i}\Big\|_{(L_1+L_2)(0,\infty)},\;\; n=1,2,\dots$$
Here, $\sum_{i=1}^n \oplus{x_i}$ is the disjoint sum of $x_i,$ $i=1,2,\dots$ (see Section \ref{aux1}), which can be thought as the function $\sum_{i=1}^n \overline{x_i}$, where $\overline{x_i}=f_i(t-i+1)\chi_{[i-1,i)}(t)$, $t>0$.

Since $\overline{u_i}=\overline{x_i}+\overline{y_i}$ and $m({\rm supp}\,\overline{y_i})\le 1$, then
$$
\|\overline{u_i}-\overline{x_i}\|_{(L_1+L_2)(0,\infty)}=\|\overline{y_i}\|_{(L_1+L_2)(0,\infty)}=\|{y_i}\|_{1},\;\;i=1,2,\dots.$$
Hence, denoting $z:=u_1^*$ and taking into account that $\|y_i\|_1\rightarrow 0$, in the same manner as above (passing, if necessary, to some further subsequence), we get
$$
\Big\|\sum_{i=1}^nu_i\Big\|_1\leq C_2\Big\|\sum_{i=1}^n\overline{u_i}\Big\|_{(L_1+L_2)(0,\infty)}.
$$
Note that the functions
$$
\sum_{i=1}^n\overline{u_i}(t)\;\;\mbox{and}\;\;\sum_{i=1}^nz(t-i+1)\chi_{(i-1,i)}(t)$$
are equimeasurable on $(0,\infty)$. Combining this together with the previous inequality and the definition of the norm in $(L_1+L_2)(0,\infty)$, we obtain that
\begin{equation}
\label{one more est}
\Big\|\sum_{i=1}^nu_i\Big\|_1\leq C_2\left(n\int_0^{1/n}{z}(s)\,ds+\left(n\int_{1/n}^1{z}(s)^2\,ds\right)^{1/2}\right),\;\; n=1,2,\dots,
\end{equation}

On the other hand, since $h_i$, $i=1,2,\dots$, are pairwise disjoint, we can assume (if necessary, passing to a subsequence) that the sequence $\{h_i\}$ is equivalent in $L_M$ to the unit vector basis $\{e_i\}$ in some Orlicz sequence space $l_\psi$, where $\psi\in C_M^\infty$ (see, for instance, \cite[Proposition~3]{LTIII}). Thus, in view of \eqref{equiv_1l_M}, \eqref{one more est} and the equation 
$$
\Big\|\sum_{k=1}^n e_k\Big\|_{l_\psi}=\frac{1}{\psi^{-1}(1/n)},\;\;n=1,2,\dots,
$$ 
we get the estimate 
$$
\frac{1}{\psi^{-1}(1/n)}\leq C_3\left(n\int_0^{1/n}{z}(s)\,ds+\left(n\int_{1/n}^1{z}(s)^2\,ds\right)^{1/2}\right),\;\;n=1,2,\dots,$$
or, by convexity of $\psi$,
\begin{equation}
\label{corr_M_g}
\frac{1}{\psi^{-1}(t)}\leq C\left(\frac1t\int_0^t{z}(s)\,ds+\left(\frac1t\int_t^1{z}(s)^2\,ds\right)^{1/2}\right),\;\; 0<t\leq 1,
\end{equation}
for some $C>0$.
Since $1<\alpha_M^\infty\le \beta_M^\infty<2$, then, by \cite{Mont-94} (see also \cite[\S\,II.8.6]{KPS} and \cite[p.~105]{KMP-07}), the operators $T_1$ and $T_2$ defined by 
$$
T_1x(t):=\frac1t\int_0^t{x}(s)\,ds\;\;\mbox{and}\;\;T_2x(t):=\left(\frac1t\int_t^1{x}(s)^2\,ds\right)^{1/2},\;\; 0<t\leq 1,$$
are bounded in $L_M$. Therefore, recalling that $z\in L_M$, we see that the function from the right-hand side of \eqref{corr_M_g} belongs to $L_M$ as well. Thus, by \eqref{corr_M_g}, ${1}/{\psi^{-1}}\in L_M$. Since $\psi\in C_M^\infty$, the proof of the proposition is completed.
\end{proof}

Now, we are ready to prove the first main result of the paper, an  extension of Rosenthal's theorem for $L^p$-spaces (see \cite[Theorem~13]{Ros}, \cite[Theorem~7.2.6]{AK} or Section \ref{Intro}), to the family of Orlicz spaces.

\begin{theor}
\label{theorem-main}
Let $M$ be an Orlicz function such that $1<\alpha_M^\infty\le \beta_M^\infty<2$ and $1/\psi^{-1}\not\in L_M$ for every $\psi\in C_{M}^{\infty}$. Suppose that $H$ is a subspace of $L_M$. The following conditions are equivalent:

(i) $H$ does not contain any infinite-dimensional subspace isomorphic to a subspace spanned in $L_M$ by pairwise disjoint functions (equivalently, to an Orlicz sequence space $\ell_\psi$, with ${\psi}\in C_M^\infty$);

(ii) $H$ is a strongly embedded subspace of $L_M$;

(iii) the unit ball $B_H$ of $H$ has equi-absolutely continuous norms in $L_M$.

In particular, conditions (i), (ii) and (iii) are equivalent provided if $M$ is an Orlicz function such that $1<\alpha_M^\infty\le \beta_M^\infty<2$ and it satisfies at least one of the conditions  \eqref{eq0} or \eqref{eq00}.
\end{theor}

\begin{proof}
(i)$\Longrightarrow$ (ii).
If $H$ fails to be a strongly embedded subspace of $L_M$, then there is a sequence $\{f_n\}\subset H$, $\|f_n\|_{L_M}=1$, $n=1,2,\dots$ such that norms of $L_M$ and $L^1$ are not equivalent on $H$ (see e.g. \cite[Proposition~6.4.5]{AK}). Hence, there is a subsequence $\{f_n'\}\subset \{f_n\}$, which is equivalent in $L_M$ to the unit vector  basis of some Orlicz sequence space $\ell_{\psi}$, with ${\psi}\in C_M^\infty$ \cite[Proposition~3]{LTIII}. Since $\ell_{\psi}$ may be spanned in $L_M$ by pairwise disjoint functions \cite[Proposition~4]{LTIII},
this contradicts the hypothesis of this implication.

(ii)$\Longrightarrow$ (i).
Suppose that $H$ is a strongly embedded subspace of $L_M$, but there is a sequence $\{g_n\}\subset H$, which is equivalent in $L_M$ to a sequence of pairwise disjoint functions. Then, again by \cite[Proposition~3]{LTIII}, we can assume (passing to a subsequence if necessary) that $\{g_n\}$ is equivalent in $L_M$ to the unit vector basis of some Orlicz sequence space $\ell_{\psi}$, where ${\psi}\in C_M^\infty$. On the other hand, in view of Lemma \ref{lemma1}, we have $1<\alpha_M^\infty\le \alpha_\psi^0\le \beta_\psi^0\le\beta_M^\infty<2$, which implies, by Lemma \ref{Lemma 20}, that $\psi$ is equivalent to some $(1+\varepsilon)$-convex and $(2-\varepsilon)$-concave Orlicz function for small values of the argument and some $\varepsilon>0$. Therefore, since $H$ is strongly embedded in $L_M$, from \cite[Corollary~3.3]{A-16} it follows that ${1}/{{\psi}^{-1}}\in L_M$. Since ${\psi}\in C_M^\infty$, this contradicts the hypothesis of the theorem.

Combining the condition that $1/\psi^{-1}\not\in L_M$ for every $\psi\in C_{M}^{\infty}$ and Proposition \ref{prop-main}, we see that (ii) implies (iii). Therefore, since the converse implication (iii)$\Longrightarrow$ (ii) is obvious, the proof of the first assertion is completed.

Since now the second assertion is an immediate consequence of Proposition \ref{prop1}, the theorem is proved.
\end{proof}





Recall that a r.i.\ space $X$ is said to be binary whenever the characteristic $\eta_X(H)$ (see Section~\ref{prel3}) takes on closed linear subspaces $H\subset X$ only two values, $0$ and $1$. Clearly, Theorem \ref{theorem-main} implies the following.

\begin{cor}
\label{Cor2}
If $M$ is an Orlicz function such that $1<\alpha_M^\infty\le \beta_M^\infty<2$ and $1/\psi^{-1}\not\in L_M$ for every $\psi\in C_{M}^{\infty}$, then the Orlicz space $L_M$ is binary.

In particular, $L_M$ is binary whenever $1<\alpha_M^\infty\le \beta_M^\infty<2$ and $M$ satisfies at least one of the conditions \eqref{eq0} or \eqref{eq00}.
\end{cor}


Let $L_M$ be an Orlicz space with $1<\alpha_M^\infty\le \beta_M^\infty<2$. From Theorem \ref{theorem-main} it follows that $L_M$ satisfies the extension of Rosenthal's theorem, given in Theorem \ref{theorem-main}, whenever $1/\psi^{-1}\not\in L_M$ for every $\psi\in C_{M}^{\infty}$. Conversely, what we can say about an Orlicz function $M$ if we know that this extension holds for the Orlicz space $L_M$? 

To this end, the following result will be useful.

\begin{prop}
\label{Prop1a}
Let a r.i.\ space $X$ contain two sequences $\{x_n\}_{n=1}^\infty$ and $\{y_n\}_{n=1}^\infty$ satisfying the following conditions:

(a) $x_n$, $n=1,2,\dots$, are pairwise disjoint;

(b) $\eta_{X}([y_n])=0$;

(c) there is a constant $C>0$ such that for all $c_n\in\mathbb{R}$ we have
$$
\Big\|\sum_{n=1}^\infty c_nx_n\Big\|_X\leq C\Big\|\sum_{n=1}^\infty c_ny_n\Big\|_X.$$

Then, the space $X$ is not binary, i.e., there is a subspace $H$ of $X$ with $\eta_{X}(H)\in(0,1)$. Every such a subspace $H$ is strongly embedded in $X$, while its unit ball $B_H$ has non-equi-absolutely continuous norms in $X$.

If instead of (c) the stronger condition

(c') $\{x_n\}_{n=1}^\infty$ and $\{y_n\}_{n=1}^\infty$ are equivalent in $X$ to the unit vector basis of some Orlicz space $\ell_ \psi$

holds, then additionally $H$ is isomorphic to the space $\ell_\psi$.
\end{prop}

\begin{proof}
We use some arguments from the proof of Theorem~2 in \cite{Nov-sb}.

Since $X$ is a r.i.\ space, the dilation operator $\sigma_\tau$ is bounded in $X$ for each $\tau>0$ (see Section \ref{prel1}). Hence, without loss of generality, we may assume that ${\rm supp}\,x_n\subset [0,1/2]$ and ${\rm supp}\,y_n\subset [1/2,1]$, $n=1,2,\dots$. Let $\lambda \in(0,1)$ (it will be specified later), $x_n':=\lambda x_n$,  $z_n:=x_n'+y_n$, $n=1,2,\dots$, and $H_\lambda:=[z_n]$. 

On the one hand, since $\lim_{n\to\infty}m({\rm supp}\,x_n)=0$, for every $\lambda \in(0,1)$ we have
\begin{equation}\label{lower estimate}
\eta_{X}(H_\lambda)\ge \inf_{n=1,2,\dots}\frac{\|z_n\chi_{{\rm supp}\,x_n}\|_X}{\|z_n\|_X}\ge \frac{\lambda}{\lambda+1}>0.
\end{equation}
On the other hand, from condition (b) it follows that
$$
\eta_{X}(H_\lambda)\le \sup_{z\in H_\lambda}\frac{\|z\chi_{[0,1/2]}\|_X}{\|z\|_X}.
$$
If $z=\sum_{n=1}^\infty c_nz_n$, then, by (c),
$$
\|z\chi_{[0,1/2]}\|_X=\lambda \Big\|\sum_{n=1}^\infty c_nx_n\Big\|_X\le C \lambda \Big\|\sum_{n=1}^\infty c_ny_n\Big\|_X=C\lambda\|z\chi_{[1/2,1]}\|_X.$$
Consequently, $\eta_{X}(H_\lambda)\le C\lambda$. Thus, $\eta_{X}(H_\lambda)<1$ whenever $\lambda<1/C$. Combining this together with estimate \eqref{lower estimate}, for each $\lambda<1/C$, we get $\eta_{X}(H_\lambda)\in(0,1)$. Thus, the subspace $H_\lambda$ is strongly embedded in $X$, while its unit ball $B_H$ has non-equi-absolutely continuous norms in $X$. Hence, the proof of the first assertion of the proposition is completed. As concerns the second one, it is a direct consequence of the definition of the subspace $H_\lambda$.
\end{proof}

\begin{theor}
\label{theorem-main2}
Let $M$ be an Orlicz function such that $M\in \Delta_2^\infty$, $\tilde{M}\in \Delta_2^\infty$ and there exists a function $\psi\in C_{M}^{\infty}$ satisfying the conditions: for some $C>0$ and all $0\le u,v\le 1$
\begin{equation}
\label{submult}
\psi(uv)\le C\psi(u)\psi(v)
\end{equation}
and $1/\psi^{-1}\in L_M$.

Then there is a strongly embedded subspace $H$ of the Orlicz space $L_M$ such that the unit ball $B_H$ of $H$ fails to have equi-absolutely continuous norms in $L_M$.
\end{theor}

\begin{proof}
First, since $\psi\in C_{M}^{\infty}$, then there is a sequence $\{x_n\}_{n=1}^\infty$ of pairwise disjoint functions which is equivalent to the unit vector basis in the Orlicz space $\ell_\psi$ (see Section \ref{prel2}).

Let us consider the Marcinkiewicz space $M(\varphi)$, where $\varphi(t):=t/\psi^{-1}(t)$, $0\le t\le 1$ (see Section \ref{prel1}). Since $\tilde{M}\in \Delta_2^\infty$, then $\alpha_M^\infty>1$. Consequently, by Lemmas \ref{lemma1} and \ref{Lemma 20}, the function $\psi$ is $(1+\varepsilon)$-convex for some $\varepsilon>0$ and small values of the argument. Then, applying \cite[Theorem~II.5.3]{KPS} (see also Lemma 3.6 in \cite{A-16}), one can easily check that 
$$
\|x\|_{M(\varphi)}\asymp \sup_{0<t\le 1}\frac{x^*(t)}{\varphi(t)}=\sup_{0<t\le 1}x^*(t)\psi^{-1}(t).$$

Then, in view of hypothesis \eqref{submult}, from \cite[Theorem~3.8]{A-16} (see also \cite[Theorem~4.1]{ASZ-22}) it follows that any sequence $\{y_n\}_{n=1}^\infty$ of mean zero independent functions equimeasurable with the function $1/\psi^{-1}$ is equivalent in every r.i. space $X$ such that $1/\psi^{-1}\in X$ to the unit vector basis in $\ell_\psi$. In particular, this holds for $X=L_M$. Furthermore, applying Lemma \ref{lemma 3}, we find a function $N(u)$ equivalent to some Orlicz function such that ${N}\in \Delta_2^\infty$, $\tilde{N}\in \Delta_2^\infty$, $1/\psi^{-1}\in L_N$ and $\lim_{u\to\infty}N(u)/M(u)=\infty.$ Consequently, as above, the sequence $\{y_n\}_{n=1}^\infty$ is equivalent in $L_N$ to the unit vector basis in $\ell_\psi$ as well. Thus, for the subspace $H:=[y_n]_{L_M}$ we have $\sup_{x\in B_H}\|x\|_{L_N}<\infty$,  and hence, by the de la Vall\'{e}e Poussin description of sets having equi-absolutely continuous norms in r.i.\ spaces \cite[Theorem~3.2]{LMT} (see also \cite[Lemma~4]{A-14} or \cite[Theorem~5.1]{CFMN}), the unit ball $B_H$ of $H$ has equi-absolutely continuous norms in $L_M$. Equivalently, by Lemma \ref{lemma 2}, we have $\eta_{X}([y_n])=0$.

As a result, the space $X=L_M$ and the sequences $\{x_n\}_{n=1}^\infty$, $\{y_n\}_{n=1}^\infty$ satisfy conditions (a), (b) and (c') of Proposition \ref{Prop1a}. Applying it, we obtain the assertion of the theorem.
 
\end{proof}

\begin{cor}
\label{cor-main1}
Let $M$ be an Orlicz function such that $M\in \Delta_2^\infty$, $\tilde{M}\in \Delta_2^\infty$ and $t^{-1/\beta_M^\infty}\in L_M$. Then,  there is a strongly embedded subspace $H$ of $L_M$ such that the unit ball $B_H$ of $H$ fails to have equi-absolutely continuous norms in $L_M$.
\end{cor}

\begin{proof}
By \cite[Proposition~4]{LTIII}, $t^{\beta_M^\infty}\in C_M^\infty$, and so there is a sequence $\{x_n\}_{n=1}^\infty$ of pairwise disjoint functions which spans $\ell^{\beta_M^\infty}$ in $L_M$. Therefore, applying Theorem \ref{theorem-main2} for $\psi(t)=t^{\beta_M^\infty}$, we come to the required result\footnote{Observe that in this special case the role of the Marcinkiewicz space $M(\varphi)$ will be played by the "weak"\:$L^p$-space $L_{p,\infty}$ endowed with the quasi-norm
$\|x\|_{L_{p,\infty}}:=\sup_{0<t\le 1}x^*(t)t^{1/p}.$ Also, instead of the above general results from \cite{A-16} or \cite{ASZ-22} we may make use of \cite[Theorem~2]{Brav} (see also \cite[Theorem~3, p.~51]{Brav2}).}.

\end{proof}





If we confine ourselves to considering only subspaces of an Orlicz space that are isomorphic to Orlicz sequence spaces, we come to the following simpler criterion.

\begin{theor}
\label{theorem-main5}
Let $M$ be an Orlicz function such that $1<\alpha_M^\infty\le \beta_M^\infty<2$. Then, the unit ball $B_H$ of an arbitrary strongly embedded subspace $H$ of the Orlicz space $L_M$, which is isomorphic to an Orlicz sequence space, has equi-absolutely continuous norms in $L_M$ if and only if $t^{-1/\beta_M^\infty}\not\in L_M$.
%
%
%
\end{theor}

\begin{proof}
In view of Theorem \ref{theorem-main2}, it suffices to prove that the unit ball $B_H$ of an arbitrary strongly embedded subspace $H$ of the Orlicz space $L_M$, which is isomorphic to an Orlicz sequence space, has equi-absolutely continuous norms in $L_M$ whenever $t^{-1/\beta_M^\infty}\not\in L_M$.

On the contrary, assume that there exists a strongly embedded subspace $H$ of $L_M$, isomorphic to some Orlicz sequence space $\ell_\varphi$ and having non-equi-absolutely continuous norms. Arguing in the same way as in the proof of Proposition \ref{prop-main}, we construct a weakly null unconditional basic sequence $\{f_i\}_{i=1}^\infty\subset H$ and a sequence $\{F_i\}_{i=1}^\infty$ of some pairwise disjoint subsets of $[0,1]$ such that estimate \eqref{extra import} holds. Since $H$ is isomorphic to $l_\varphi$, we can assume (passing to a further subsequence if necessary) that $\{f_i\}$ is equivalent in $L_M$ to a block basis of the unit vector basis of $l_\varphi$. In turn, each such a block basis contains a subsequence, which is equivalent to the unit vector basis of some Orlicz sequence space $l_\theta$
(see e.g. \cite[p.~141]{LT77}). Thus, we can assume (possibly, again passing to a subsequence) that $[f_i]$ is itself isomorphic to the above basis. Observe also, for the future, that $[f_i]$ as a subspace of $H$ is strongly embedded in $L_M$.
  
Furthermore, as above, we can assume that the sequence $\{h_i\}$, where $h_i:=f_i\chi_{F_i}$, $i=1,2,\dots$, is equivalent in $L_M$ to the unit vector basis of an Orlicz sequence space $l_\psi$, where $\psi\in C_M^\infty$ (see, for instance, \cite[Proposition~3]{LTIII}). Now, repeating the arguments from the proof of Proposition \ref{prop-main}, we have
$$
\|(c_k)\|_{l_\theta}\asymp\Big\|\sum_{k=1}^\infty c_kf_k\Big\|_{L_M} \ge \frac{1}{\sqrt{2}}\Big\|\sum_{k=1}^\infty c_kh_k\Big\|_{L_M}\asymp \|(c_k)\|_{l_\psi},
$$ 
and hence for some $C>0$
\begin{equation}\label{one side}
\psi(t)\le C\theta(t),\;\;0<t\le 1.
\end{equation}

Let us assume now that $\beta_M^\infty< \alpha_\theta^0$. Then, by Lemma \ref{lemma1}, we get $\beta_\psi^0< \alpha_\theta^0$. Then, if $q\in (\beta_\psi^0,\alpha_\theta^0)$, from the definition of the Matuszewska-Orlicz indices at zero it follows that for some $C>0$
$$
C^{-1}\theta(t)\le t^q\le C\psi(t),\;\;0<t\le 1.$$
Combining this with inequality \eqref{one side}, we conclude that the functions $\theta$ and $\psi$ are equivalent on $(0,1]$. Thus, in particular, $\alpha_\psi^0=\alpha_\theta^0$ and $\beta_\psi^0=\beta_\theta^0$. Since $\beta_\psi^0\le \beta_M^\infty$ by Lemma \ref{lemma1}, this clearly contradicts the initial assumption.

Thus, $\beta_M^\infty\ge \alpha_\theta^0$. Observe that $\ell^{\alpha_\theta^0}$ is a subspace of $\ell_\theta$ \cite[Theorem~4.a.9]{LT77}, 
and hence is isomorphic to some subspace of $H$. 
Therefore, by \cite[Corollary~3.3]{A-16}, $t^{-1/\alpha_\theta^0}\in L_M$.
Since $\beta_M^\infty\ge \alpha_\theta^0$, this implies that $t^{-1/\beta_M^\infty}\in L_M$ as well. Since this contradicts the condition, the proof is completed.

\end{proof}

If an Orlicz function $M$ is regularly varying at $\infty$ of order $p$, we clearly have that $C_M^\infty=\{t^p\}$ and $\beta_M^\infty=p$. Therefore, we get

\begin{cor}
\label{reg var}
Suppose that $M$ is an Orlicz function, which regularly varies of order $p\in (1,2)$ at $\infty$. Then, the unit ball $B_H$ of an arbitrary strongly embedded subspace $H$ of the Orlicz space $L_M$ has equi-absolutely continuous norms in $L_M$ if and only if $t^{-1/p}\not\in L_M$.
\end{cor}

\begin{center}
S.V. Astashkin \\
Department of Mathematics and Mechanics\\
Samara State University\\
443011 Samara, Acad. Pavlov, 1 \\
Russian Federation\\
e-mail: {\it astash@samsu.ru}
\end{center}

\vskip 0.4cm

\end{document}